\documentclass [12 pt]{amsart}
\usepackage{latexsym}
\usepackage{amssymb}
\usepackage{amsmath}
\usepackage{amscd}
\usepackage{graphicx}
\usepackage{xypic}
\theoremstyle{plain}
\newtheorem{theorem}{Theorem}[section]

\newtheorem{corollary}[theorem]{Corollary}
\theoremstyle{definition}
\newtheorem{definition}[theorem]{Definition}
\newtheorem{proposition}[theorem]{Proposition}
\newtheorem{example}[theorem]{Example}
\theoremstyle{remark}
\newtheorem{remark}[theorem]{Remark}
\theoremstyle{definition}

\theoremstyle{remark}

\begin{document}

\title[FUZZY SUBGROUPS COMMUTATIVITY DEGREE OF DIHEDRAL GROUPS]{FUZZY SUBGROUPS COMMUTATIVITY DEGREE OF DIHEDRAL GROUPS}

\author[]{Hassan Naraghi$^{1,*}$ and Hosein Naraghi$^2$}

\address[]{Department of Mathematics, Ashtian Branch, Islamic Azad University, Iran$^1$.
\newline}
\address[]{Young Researchers and Elite Club, Ashtian Branch, Islamic Azad University, Iran$^*$.
\newline}
\address[]{Department of Mathematics, Payame Noor Uiversity, Iran$^2$.
\newline}
\email[]{naraghi@aiau.ac.ir$^1$\\
ho.naraghi@pnu.ac.ir$^2$
}



\thanks{ The authors are deeply grateful to the Islamic Azad University, Ashtian Branch for the financial support for the proposal entitled "FUZZY SUBGROUPS COMMUTATIVITY DEGREE OF DIHEDRAL GROUPS".}
\maketitle
\begin{abstract}
In this paper we introduce and study the concept of distinct fuzzy subgroups commutativity degree of a finite group $G$. This quantity measures the
probability of two  random distinct fuzzy subgroups of $G$ commuting. We determine distinct fuzzy subgroup commutativity degree for some of finite groups.

\end{abstract}
\section{Introduction}\label{s1}
In 1965, Zadeh \cite{Z} first introduced fuzzy set. Mordeson et.al (\cite{MKA1}) called him "a pioneer of work on fuzzy subsets". After that
paper, several aspects of fuzzy subsets were studied. In 1971, Rosenfled \cite{R} introduced fuzzy sets in the realm of group theory and formulated
the concepts of fuzzy subgroups of a group. An increasing number of properties from classical group theory have been generalized. In the last years
there has been a growing interest in the use of probability in finite group theory. One of the most important aspects that have been studied
is the probability that two elements of a finite groups $G$ commute. This is called the commutativity degree of $G$. Let G be a group and let $\mu$ and $\nu$ be fuzzy subgroups
of G. We say that $\mu$ is permuted by $\nu$ if for any $a,b\in
G$, there exists $x\in G$ such that
$\mu(x^{-1}ab)\geq\mu(a),\nu(x)\geq\nu(b)$ and we say $\mu$ and
$\nu$ are permutable if $\mu$ is permuted by $\nu$ and $\nu$ is
permuted by $\mu$. Also we say that $\mu$ is permuted by $\nu$
mutually if for any subgroup L of $\nu_{b}$ that $b\in Im\nu$, we
have been for any $a\in G,l\in L$, there exist $l_{1},l_{2}$ of L
such that $\mu(l^{-1}_{1}al)\geq\mu(a)$ and
$\mu(lal^{-1}_{2})\geq\mu(a)$ and we say $\mu$ and $\nu$ are
mutually permutable if $\mu$ is permuted by $\nu$ mutually and
$\nu$ is permuted by $\mu$ mutually. Let $\mu$ and $\nu$ be fuzzy
subgroups of G. In \cite{Naraghi} have been determined that $\mu$ and $\nu$ are
permutable(mutually permutable) if and if for any $t\in Im\mu,s\in
Im\nu$, $\mu_{t},\nu_{s}$ are permutable(mutually permutable) which denote by $\nu\in P(\mu)$($\nu\in MP(\mu)$). Ajmal
and Thomas \cite{Ajmal} introduced the notion of a fuzzy quasinormal
subgroup. Fuzzy quasinormal subgroup arising out of fuzzy normal
subgroup. Also in \cite{Naraghi} have been proved that $\mu$ is a fuzzy quasinormal subgroup
of group G if and only if for every subgroup L of G, we have been
that for any $a\in G,l\in L$ there exist $l_{1},l_{2}$ of L such
that $\mu(l^{-1}_{1}al)\geq\mu(a)$ and
$\mu(lal^{-1}_{2})\geq\mu(a)$. In the following, let $G$ be a finite group and denote by $F(G)$ the set of all
fuzzy subgroup of a group $G$. Let $F_{1}(G)$ be the set of all fuzzy subgroups $\mu$ of G such
that $\mu(e)=1$. In this paper, we use the natural equivalence of fuzzy subgroups
studied by Iranmanesh and Naraghi \cite{NI}. This is denoted by $\sim$ and the set of all the equivalence classes $\sim$ on $F_{1}(G)$ is
denoted by $S(G)$. We consider the quantity
$$
sd(G)=\frac{1}{|S(G)|^2}|\{(\mu,\nu)\in S(G)^2| \nu\in P(\mu)\}|
$$
which will be called the distinct fuzzy subgroup commutativity degree of $G$. Clearly, $sd(G)$ measures the probability that two
distinct fuzzy subgroups of $G$ commute. For an arbitrary finite group $G$, computing $sd(G)$ is a difficult work,
since it involves the counting of distinct fuzzy subgroups of $G$. In this paper a first step in the study of permutable fuzzy
subgroups of a finite group $G$ which in section \ref{s3} and \ref{s4} we present some basic properties and result on the permutable fuzzy
subgroups of a finite group $G$.  In the section \ref{s5} we study some basic properties and result on the natural equivalence of fuzzy subgroups
studied by Iranmanesh and Naraghi \cite{NI}. In the section \ref{s6} we determine the number of distinct fuzzy subgroups for some of dihedral groups.
In the final section deals with distinct fuzzy subgroup commutativity degree for some of finite groups.

\section{Preliminaries}\label{s2}
We use [0,1], the real unit interval as a chain the usual ordering
in R which $\wedge$ stands for infimum( or intersection) and
$\vee$ stands for supremum ( or union) for the degree of
membership. A fuzzy subset of a set X is mapping
$\mu:\rightarrow[0,1]$. The union and intersection of two fuzzy
subset are defined using sup and inf point wise. We denote the set
of all fuzzy subset of X by $I^{X}$. Further, we denote fuzzy
subsets by the Greek letters $\mu,\nu,\eta$, etc. Let $\mu,\nu\in
I^{X}$. If $\mu(x)\leq\nu(x) \forall x\in X$, then we say that
$\mu$ is contained in $\nu$ ( or $\nu$ contains $\mu$) and we
write $\mu\subseteq\nu$. Let $\mu\in I^{X}$ for $a\in I$, define
$\mu_{a}$ as follow:\\
$\mu_{a}=\{x\mid x\in X, \mu(x)\geq a\}$. $\mu_{a}$ is called
a-cut(
or a-level) set of $\mu$.\\
It is easy to verify that for any $\mu,\nu\in I^{X}$:\\
1) $\mu\subseteq\nu,a\in I\Rightarrow\mu_{a}\subseteq\nu_{a}$.\\
2) $a\leq b,a,b\in I\Rightarrow\mu_{b}\subseteq\mu_{a}$.\\
3) $\mu=\nu\Leftrightarrow\mu_{a}=\nu_{a}\forall a\in I$.\\
Let G be an arbitrary group with a multiplicative binary operation
and identity. We define the binary operation $o$ on
$I^{G}$ as follow:\\
$\forall\mu,\nu\in I^{G}$, $\forall x\in G$\\
$(\mu o \nu)(x)=\vee\{\mu(y)\wedge\nu(z)\mid y,z\in G, yz=x\}$.\\
We call $\mu o \nu$ the product of $\mu$ and $\nu$. Fuzzy subset
$\mu$ of G is called a fuzzy subgroup of G if\\
($G_{1}$) $\mu(xy)\geq\mu(x)\wedge\mu(y)\forall x,y\in G$;\\
($G_{2}$) $\mu(x^{-1}\geq\mu(x)\forall x\in G$.

\begin{proposition}\label{main}\cite[Lemma 1.2.5]{MM}.
Let $\mu\in I^{G}$. Then $\mu$ is a fuzzy
subgroup of G if and only if $\mu_{a}$ is a subgroup of G,
$\forall a\in\mu(G)\bigcup\{b\in I\mid b\leq\mu(e)\}$.
\end{proposition}

\begin{theorem}\label{t21}\cite[Theorem 1.2.9]{MM}.
Let $\mu\in I^{G}$. Then $\mu o \nu$ is a
fuzzy subgroup if and only if $\mu o \nu=\nu o \mu$.
\end{theorem}

\begin{definition}\cite{Ajmal}.
Let $\mu$ be a fuzzy subgroup of group G, $\mu$ is said to
be fuzzy normal subgroup of G if $\mu(xy)=\mu(yx)\forall x,y\in
G$.
\end{definition}

\begin{definition}\cite{Ballester}.
Let G be a group and let H and K be subgroups of G.\\
(a) We say that H and K are permutable if $HK=KH=<H,K>$.\\
(b) We say that H and K are mutually permutable if H permutes with
every subgroup of K and K permutes with every subgroup of H.
\end{definition}

\begin{definition}\cite{Ballester}.
Let G be a group and let H be a subgroup of G, H is said to
be quasinormal in G, if H permutes with every subgroup of G.
\end{definition}


\section{Permutable and mutually permutable on fuzzy subgroups of a group}\label{s3}

\begin{definition}
Let G be a group and let $\mu$ and $\nu$ be fuzzy subgroups of
G.\\
(a) We say that $\mu$ is permuted by $\nu$ if for any $a,b\in G$,
there exists $x\in G$ such that
$\mu(x^{-1}ab)\geq\mu(a),\nu(x)\geq\nu(b)$.\\
(b) We say that $\mu$ is permuted by $\nu$ mutually if for any
subgroup L of $\nu_{b}$ that $b\in Im\nu$, we have been for any
$a\in G,l\in L$, there exist $l_{1},l_{2}$ of L such that
$\mu(l^{-1}_{1}al)\geq\mu(a)$ and $\mu(lal^{-1}_{2})\geq\mu(a)$.
\end{definition}

\begin{definition}
Let G be a group and let $\mu$ and $\nu$ be fuzzy subgroups of
G.\\
(a) We say $\mu$ and $\nu$ are permutable if $\mu$ is permuted by
$\nu$ and $\nu$ is permuted by $\mu$.\\
(b) We say $\mu$ and $\nu$ are mutually permutable if $\mu$ is
permuted by $\nu$ mutually and $\nu$ is permuted by $\mu$
mutually.
\end{definition}

\begin{corollary}\label{c31}
Let $\mu$ and $\nu$ be fuzzy subgroups of G. If $\mu$ and $\nu$
are mutually permutable then $\mu$ and $\nu$ are permutable.
\end{corollary}

\begin{proof}
Straightforward.
\end{proof}

\begin{corollary}\label{c32}
Let $\mu$ is a fuzzy normal subgroup of G. Then $\mu$ permutes
with every fuzzy subgroup of G mutually.
\end{corollary}

\begin{proof}
Straightforward.
\end{proof}

\begin{theorem}\label{t31}
Let $\mu$ and $\nu$ be fuzzy subgroups of G, then $\mu$ and $\nu$
are permutable if and if for any $t\in Im\mu,s\in Im\nu$,
$\mu_{t},\nu_{s}$ are permutable.
\end{theorem}

\begin{proof}
Let $\mu$ and $\nu$ be permutable. Let $t\in Im\mu,s\in Im\nu$. If
$a\in\mu_{t}$ and $b\in\nu_{s}$ then $\mu(a)\geq t,\nu(b)\geq s$.
We know that $\mu$ is permuted by $\nu$. Then there that exists
$x\in G$ such that $\mu(x^{-1}ab)\geq t$ and $\nu(x)\geq s$, this
means that $x^{-1}ab\in\mu_{t}$ and $x\in\nu_{s}$. So that
$ab=x(x^{-1}ab)$. If $a\in\nu_{s},b\in\mu_{t}$, then $\mu(b)\geq
t,\nu(a)\geq s$. So that there exists $y\in G$ such that
$\nu(y^{-1}ab)\geq\nu(a)\geq s$ and $\mu(y)\geq\mu(b)\geq t$, this
means that $y^{-1}ab\in\nu_{s}$ and $y\in\mu_{t}$. So that
$ab=y(y^{-1}ab)$, consequently $\mu_{t}\nu_{s}=\nu_{s}\mu_{t}$.
Now let $\mu_{t}\nu_{s}=\nu_{s}\mu_{t},\forall t\in Im\mu,s\in
Im\nu$ and let a and b be two arbitrary elements of G. Let
$r=\mu(a),s=\nu(b)$, then elements exist for example
$a'\in\mu_{t},b'\in\nu_{s}$ such that $ab=a'b'$, then
$b'^{-1}ab=a'$, this implies $\mu(b'^{-1}ab)=\mu(a')\geq
t=\mu(a)$. Hence $b'\in\nu_{s}$, then $\nu(b')\geq s=\nu(b)$.
Therefore $\mu$ is permuted by $\nu$. Similarly $\nu$ is permuted
by $\mu$.
\end{proof}

\begin{proposition}\label{p31}
Let $\mu$ and $\nu$ be fuzzy subgroups of G and $t\in Im\mu,s\in
Im\nu$ if $\mu$ and $\nu$ be permutable then\\
(1) If $t\leq s$ then there exists $a\in G$ such that
$\nu(a)\geq t$.\\
(2) If $s\leq t$ then there exists $b\in G$ such that $\mu(b)\geq
s$.
\end{proposition}

\begin{proof}
We know that $\mu_{t},\nu_{s}\neq\emptyset$ then there exist a and
b in G such that $\mu(a)\geq t$ and $\nu(b)\geq s$. Hence $\mu$
and $\nu$ are permutable then $\mu_{t}\nu_{s}=\nu_{s}\mu_{t}$,
then there are $a'\in\mu_{t}$ and $b'\in\nu_{s}$ such that
$ab=a'b'$. Therefore $\mu(aa')\geq\min\{\mu(a),\mu(a')\}\geq t$.
Similarly $\nu(bb')\geq s$. If $t\leq s$ then $\nu(bb')\geq s\geq
t$ and if $s\leq t$ then $\mu(aa')\geq t\geq s$.
\end{proof}

\begin{proposition}\label{p32}
Let $\mu$ and $\nu$ be fuzzy subgroups of G. If $\mu$ and $\nu$ be
permutable then $\mu o \nu$ is a fuzzy subgroup of G.
\end{proposition}

\begin{proof}
Let $\mu$ and $\nu$ be permutable and $x\in G$. If $y\in G$ be an
arbitrary element then there exists $t\in G$ such that
$\mu(t^{-1}yy^{-1}x)\geq\mu(y)$ and $\nu(t)\geq\nu(y^{-1}x)$, so
that $\mu(y)\wedge\nu(y^{-1}x)\leq\mu(t^{-1}x)\wedge\nu(t)$.
Therefore $\mu(y)\wedge\nu(y^{-1}x)\leq\sup_{z\in G}
\{\nu(z)\wedge\mu(z^{-1}\}$, means that $(\mu o
\nu)(x)\leq(\nu o \mu)(x)$. Similarly $(\nu o \mu)(x)\leq(\mu o
\nu)(x)$ because $\nu$ is permuted by $\mu$.
\end{proof}

\begin{example}
Let G be symmetric group $S_{3}$. Define $\mu$ and $\nu$ as follow:\\
$ \mu(x)=\left\{
\begin{matrix}
  1 & x=e \\
  \frac{1}{2} & x=b \\
  \frac{1}{3} & else
\end{matrix}
\right. ,~~~~~~~~~ $ $ \nu(x)=\left\{
\begin{matrix}
  1 & x=e \\
  \frac{1}{2} & x=ab \\
  \frac{1}{3} & else
\end{matrix}
\right.
$\\
Clearly, $\mu o \nu=\mu$, but $\mu$ is not permuted by $\nu$.
\end{example}

\begin{theorem}\label{t32}
Let $\mu$ and $\nu$ be fuzzy subgroups of G, then $\mu$ and $\nu$
are mutually permutable if and if for any $t\in Im\mu,s\in Im\nu$,
$\mu_{t},\nu_{s}$ are mutually permutable.
\end{theorem}

\begin{proof}
Let $\mu$ and $\nu$ be mutually permutable. Let $a\in Im\mu$ and
$b\in Im\nu$. Also let $L\leq\nu_{b}, x\in\mu_{a}$ and $l\in L$,
then $\mu(x)\geq a$. We known that exists $l_{1}\in L$ such that
$\mu(l_{1}^{-1}xl)\geq\mu(x)$, this means that
$l_{1}^{-1}xl\in\mu_{a}$, so that $xl=l_{1}(l_{1}^{-1}xl)$.
Therefore $\mu_{a}L\subseteq L\mu_{a}$ and also there exists
$l_{2}\in L$ such that $\mu(lxl_{2}^{-1})\geq\mu(x)\geq a$. That
is, $lxl_{2}^{-1}\in\mu_{a}$. So that $lx=(lxl_{2}^{-1})l_{2}$,
therefore $L\mu_{a}\subseteq\mu_{a}L$. So $\mu_{a}L$ is a subgroup
of G. Similarly, we know that $\nu$ is permuted by $\mu$ mutually
then for any subgroup H of $\mu_{a}$, $H\nu_{b}=\nu_{b}H$. So
$\mu_{a}$ and $\nu_{b}$ are mutually permutable. Now let for any
$a\in Im\mu$ and $b\in Im\nu$, $\mu_{a}$ and $\nu_{b}$ be mutually
permutable. Let $b\in Im\nu$ and $L\leq\nu_{b}$ and also $x\in G$
and $l\in L$. Let $r=\mu(x)$, so that $\mu_{r}$ and $\nu_{b}$ are
mutually permutable, therefore exist $l_{1}\in L$ and
$y\in\mu_{r}$ such that $lx=yl_{1}$, then $lxl_{1}^{-1}=y$, this
implies $lxl_{1}^{-1}\in\mu_{r}$ and $\mu(lxl_{1}^{-1})\geq
r=\mu(x)$. Also there exist $l_{2}\in L$ and $y'\in\mu_{r}$ such
that $xl=l_{2}y'$, then $l_{2}^{-1}xl=y'$, this implies
$l_{2}^{-1}xl\in\mu_{r}$ and $\mu(l_{2}^{-1}xl)\geq\mu(x)$.
Therefore $\mu$ is permuted by $\nu$ mutually. Similarly $\nu$ is
permuted by $\mu$ mutually.
\end{proof}


\section{Some properties of fuzzy quasinormal subgroup of a group}\label{s4}

\begin{definition}\cite{MKA}.
A fuzzy subgroup $\mu$ of G is called quasinormal if its level
subgroups are quasinormal subgroups of G.
\end{definition}

\begin{theorem}\label{main}
If $\mu$ is a fuzzy subgroup of group G, then the following
properties are equivalent:\\
($q_{1}$) For every subgroup L of G, we have been that for any
$a\in G,l\in L$ there exist $l_{1},l_{2}$ of L such that
$\mu(l^{-1}_{1}al)\geq\mu(a)$ and $\mu(lal^{-1}_{2})\geq\mu(a)$.
($q_{2}$) For any $a\in Im\mu$, $\mu_{a}$ is a quasinormal
subgroup of G.
\end{theorem}

\begin{proof}
Assume firstly the validity of ($q_{1}$). Let $a\in Im\mu$ and
$L\leq G$. If $x\in\mu_{a}, l\in L$ then there exists $l_{1}\in L$
such that $\mu(l_{1}^{-1}xl)\geq\mu(x)\geq a$, this means that
$l_{1}^{-1}xl\in\mu_{a}$. So that $xl=l_{1}(l_{1}^{-1}xl)$. Also
let $y\in\mu_{a},l'\in L$, therefore there exists $l_{2}\in L$
such that $\mu(l'yl_{2}^{-1})\geq\mu(y)$. So
$\mu(l'yl_{2}^{-1})\geq a$, this means that
$l'yl_{2}^{-1}\in\mu_{a}$, Therefore $l'y=(l'yl_{2}^{-1})l_{2}$,
consequently $L\mu_{a}=\mu_{a}L$. Hence ($q_{1}$) implies
($q_{2}$). Assume next the validity of ($q_{2}$). Let $L\subseteq
G$ and $x\in G,l\in L$. If $r=\mu(x)$ then there exist
$y\in\mu_{r}$ and $l_{1}\in L$ such that $xl=l_{1}y$, so
$\mu(l_{1}^{-1}xl)\geq r=\mu(x)$. Similarly there exist
$y'\in\mu_{r},l_{2}\in L$ such that $ix=y'l_{2}$. Then
$\mu(lxl_{2}^{-1})\geq\mu(x)$. Hence ($q_{2}$) implies ($q_{1}$).
\end{proof}

\begin{corollary}\label{main}
Let $\mu$ be a fuzzy subgroup of G. Then $\mu$ is a fuzzy
quasinormal subgroup if and only if for every subgroup L of G, we
have been that for any $a\in G,l\in L$ there exist $l_{1},l_{2}$
of L such that $\mu(l^{-1}_{1}al)\geq\mu(a)$ and
$\mu(lal^{-1}_{2})\geq\mu(a)$.
\end{corollary}
\begin{proof}
Straightforward.
\end{proof}

\begin{theorem}\label{main}\cite[Theorem 4.3.13]{MKA}.
Let $\mu$ be a fuzzy subgroup of G with finite image. Then $\mu$
is fuzzy quasinormal if and only if $\mu o \nu=\nu o \mu$, for all
fuzzy subgroups $\nu$ of group G.
\end{theorem}

\begin{corollary}\label{main}
Let $\mu$ be a fuzzy subgroup of G with finite image. Then $\mu o
\nu=\nu o \mu$, for all fuzzy subgroups $\nu$ of group G if and
only if for every subgroup L of G, we have been that for any $a\in
G,l\in L$ there exist $l_{1},l_{2}$ of L such that
$\mu(l^{-1}_{1}al)\geq\mu(a)$ and $\mu(lal^{-1}_{2})\geq\mu(a)$.
\end{corollary}

\begin{proof}
Straightforward.
\end{proof}

\begin{corollary}\label{main}
Let $\mu$ be a fuzzy normal subgroup of group G. Then $\mu$ is
fuzzy quasinormal subgroup of G.
\end{corollary}

\begin{proof}
Straightforward.
\end{proof}

\begin{corollary}\label{c4}
Let $\mu$ be a fuzzy quasinormal subgroup of group G. Then $\mu$ is
permuted by every fuzzy subgroup of G.
\end{corollary}

\begin{proof}
Straightforward.
\end{proof}

\section{on the natural equivalence of fuzzy subgroups of a finite group}\label{s5}
Whenever possible we follow the notation and terminology of \cite{NI}.\\
The dihedral group $D_{2n}$ ($n \geq 2$) is the symmetry group of a regular polygon with $n$ sides and it
has the order $2n$. The most convenient abstract description of $D_{2n}$ is obtained by using its generators:\\
a rotation $\alpha$ of order $n$ and a reflection $\beta$ of order 2. Under these notations, we have
$$
D_{2n}=\langle \alpha,\beta | \alpha^n=\beta^2=1, \beta\alpha\beta=\alpha^{-1}\rangle.
$$
\begin{definition}
Let $G$ be a group and $\mu \in F(G)$. The set $\{x\in
G|\mu(x)>0\}$ is called the support of $\mu$ and denoted by
$supp\mu$.
\end{definition}
Let $G$ be a group and $\mu \in F(G)$. We shall write $Im\mu$
for the image set of $\mu$ and $F_{\mu}$ for the family
$\{\mu_{t}|t\in Im\mu\}$.
\begin{theorem}\cite{ZK}\label{t51}.
Let $G$ be a fuzzy group. If $\mu$ is a fuzzy subset of $G$, then $\mu\in F(G)$ if and only if for all $\mu_{t}\in F_{\mu}$, $\mu_{t}$ is a subgroup of $G$.
\end{theorem}
Let $F_{1}(G)$ be the set of all fuzzy subgroups $\mu$ of G such
that $\mu(e)=1$ and let $\sim_{R}$ be an equivalence relation on
$F_{1}(G)$. We denote the set $\{\nu \in F_{1}(G)|\nu
\sim_{_{R}}\mu\}$  by $\frac{\mu}{\sim_{_{R}}}$ and the set
$\{\frac{\mu}{\sim_{_{R}}}|\mu\in F_{1}(G)\}$  by
$\frac{F_{1}(G)}{\sim_{_{R}}}$.
\begin{definition}\cite{MKA}.
Let $G$ be a group, and $\mu ,\nu\in F(G)$. $\mu$ is equivalent to $\nu$, written as  $\mu \sim \nu$ if
\begin{enumerate}
  \item $\mu(x) > \mu(y) \Leftrightarrow \nu(x)>\nu(y)$ for all $x,y \in G$.
  \item $\mu(x)=0 \Leftrightarrow  \nu(x)=0$ for all $x \in G$.
\end{enumerate}
\end{definition}
The number of the equivalence classes $\sim$ on $F_{1}(G)$ is
denoted by $s(G)$. We means the number of distinct fuzzy subgroups of $G$ is $s(G)$.
\begin{theorem}\label{t52}\cite{NI}.
Let $G$ be a finite group. The
number of  distinct fuzzy subgroups of $G$ such that their
support is exactly equal to $G$ is
$\frac{s(G)+1}{2}$.
\end{theorem}
\begin{proof}
Let
$$ U(G) = \{\frac{\mu}{\sim}|\mu\neq\mu^{*},\mu\in F(G),supp\mu=G\}$$
where $\mu^*$ is a fuzzy subgroup of $G$ and $\mu^*(x) = 1$ for all $x\in G$.
$$
V(G) = \{\frac{\mu}{\sim}|\mu\in F_1(G),supp\mu\subset G\}.
$$
Since $G$ is finite, we can define $\frac{\mu}{\sim}$ as follow:
$$
\frac{\mu}{\sim}=\frac{(\overbrace{1\cdots 1}^{n'_0}\overbrace{\lambda_1\cdots\lambda_1}^{n'_1}\cdots\overbrace{\lambda_r\cdots\lambda_r}^{n'_r})\varphi}{\sim}
$$
where $Im\mu=\{1,\mu_1,\cdots,\mu_r\}$, $1>\lambda_1>\cdots>\lambda_r>0$ and\\
$\varphi : G_0=(e)\subset G_1 \subset ... \subset G_{n_0}=\mu_1
\\
~~~~~~~~~~~~~~~~~~~~~~~~~~~~~~~~~~~~~~~~~~~~~\subset G_{n_0+1}\subset\cdots\subset G_{n_1}=\mu_{\alpha_1}
\\
~~~~~~~~~~~~~~~~~~~~~~~~~~~~~~~~~~~~~~~~~~~~~\subset G_{n_1+1} \subset\cdots\subset G_{n_2}=\mu_{\alpha_2}
\\
~~~~~~~~~~~~~~~~~~~~~~~~~~~~~~~~~~~~~~~~~~~~~~~~~\vdots
\\
~~~~~~~~~~~~~~~~~~~~~~~~~~~~~~~~~~~~~~~~~~~~~~\subset G_{n_{(r-1)}+1} \subset\cdots\subset G_{n_r}=\mu_{\alpha_r}= G
$
\\and $n'_0=n_0, n'_1=n_1-n_0$ and for all $i\in\{2,\cdots,r\}$, $n'_i=n_i-\sum\limits_{k=1}^{i-1} n_k$.\\
We define the map $f$:
$$
f:U(G)\rightarrow V(G)
$$
such that
$$
f(\frac{\mu}{\sim})=\frac{(\overbrace{1\cdots 1}^{n'_0}\overbrace{\lambda_1\cdots\lambda_1}^{n'_1}\cdots\overbrace{\lambda_{r-1}\cdots\lambda_{r-1}}^{n'_{r-1}}\overbrace{0\cdots 0}^{n'_r})\varphi}{\sim}
$$
It is easy to see that $f$ is one to one and onto. So $|U(G)|=|V(G)|$ and $s(G)=|U(G)|+|V(G)|+1$, therefore $s(G)=2|U(G)|+1$. Thus $|U(G)|=|V(G)|=\frac{s(G)-1}{2}$ and hence $|U(G)|+1=\frac{s(G)+1}{2}$.
\end{proof}
Let $G$ be a finite group. The
number of  distinct fuzzy subgroups of $G$ such that their
support is exactly equal to $G$ is denoted by $s^{\star}(G)$.
\begin{theorem}
Let $G$ be a finite group. Then the
number of  distinct fuzzy subgroups of $G$ such that their
support is exactly a subgroup of $G$ is
$\frac{s(G)-1}{2}$.
\end{theorem}
\begin{proof}
By proof of theorem \ref{t52}, $|U(G)|=|V(G)|=\frac{s(G)-1}{2}$.
\end{proof}

\begin{theorem}\label{t53}\cite{NI}.
Let $G$ be a finite group and $H$ be a  subgroup of $G$. Then the
number of  distinct fuzzy subgroups of $G$ such that their
support is exactly equal to $H$ is
$\frac{s(H)+1}{2}$.
\end{theorem}
\begin{proof}
We can easily see that the number of distinct fuzzy subgroups of the group
G which their supports is exactly H is equal to number of distinct fuzzy subgroups
of H which their supports is exactly H and this number with the previous theorem
is equal to $\frac{s(H)+1}{2}$.	
\end{proof}

\begin{corollary}\cite{NI}.
Let $G$ be a finite group and $H$ be a  subgroup of $G$. Then the
number of  distinct fuzzy subgroups of $G$ such that their
support is exactly a subgroup of $H$ is
$\frac{s(H)-1}{2}$.
\end{corollary}

\begin{proposition}\cite{MKA}\label{p51}.
Let $n\in\mathbb{N}$. Then there are $2^{n+1}-1$ distinct
equivalence classes of fuzzy subgroups of $\mathbb{Z}_{p^n}$.
\end{proposition}
\section{Counting of the distinct fuzzy subgroups for some of finite groups}\label{s6}
Now we determine the number of distinct fuzzy
subgroups for some of the dihedral groups.
\begin{example}\label{e61}
Let $G$ be the dihedral group of order 4, then $s(G)=15$.\\
\begin{figure}[]
\begin{center}
\includegraphics[angle=0, scale=0.5]
{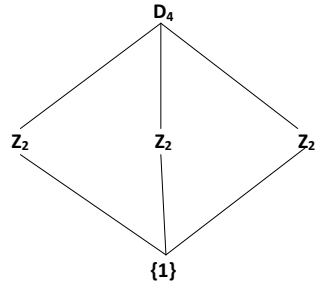} \caption{\small{Hasse diagram of $D_4$}}
\end{center}
\end{figure}
We know that $\frac{s(G)-1}{2}=s^{\star}(\{1\})+3s^{\star}(\mathbb{Z}_2)$, therefore $\frac{s(G)-1}{2}=7$. Thus $s(G)=15$.
\end{example}

\begin{example}\label{e62}
Let $G$ be the dihedral group of order 8, then $s(G)=63$.\\
\begin{figure}[]
\begin{center}
\includegraphics[angle=0, scale=0.5]
{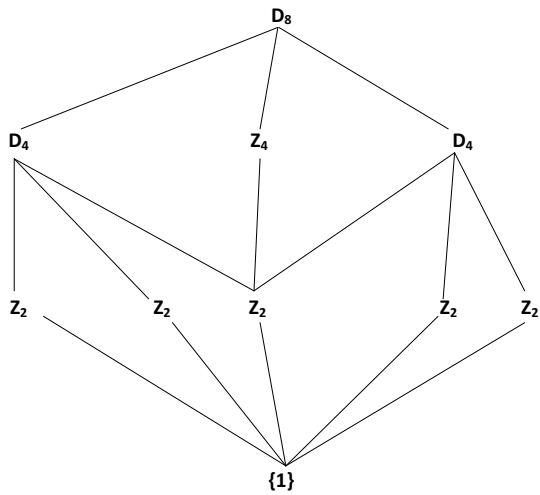} \caption{\small{Hasse diagram of $D_8$}}
\end{center}
\end{figure}
We know that $\frac{s(G)-1}{2}=s^{\star}(\{1\})+s^{\star}(\mathbb{Z}_4)+5s^{\star}(\mathbb{Z}_2)+2s^{\star}(D_4)$, therefore $\frac{s(G)-1}{2}=29$. Thus $s(G)=63$.
\end{example}
\begin{theorem}\label{t61}
Suppose that $p$ is a prime and $p\geq 3$. If $G$ is the dihedral
group of order $2p$, then $s(G)=4p+7$.
\end{theorem}
Proof. We know that $D_{2p}$ has the following maximal chains each of which can be identified with the chain $D_{2p}\supset\mathbb{Z}_{p}\supset\{0\}$ and $D_{2p}\supset\mathbb{Z}_{2}\supset\{0\}$ whose the number is $p$. Now 2 is the number of distinct fuzzy subgroups whose support is $\mathbb{Z}_{p}$, $2^1p$ is the number of distinct fuzzy subgroups whose support is $\mathbb{Z}_2$, and $2^0$ is the number of fuzzy subgroups whose support is $\{0\}$. Thus $\frac{s(G)-1}{2}=2p+2+1$, therefore $s(G)=4p+7$.

\begin{example}\label{e65}
Let $S_3=\langle \alpha, \beta | \alpha^3=\beta^2=(\alpha\beta)^2=1 \rangle$. By Hasse diagram of $S_3$,
\begin{figure}[]
\begin{center}
\includegraphics[angle=0, scale=0.5]
{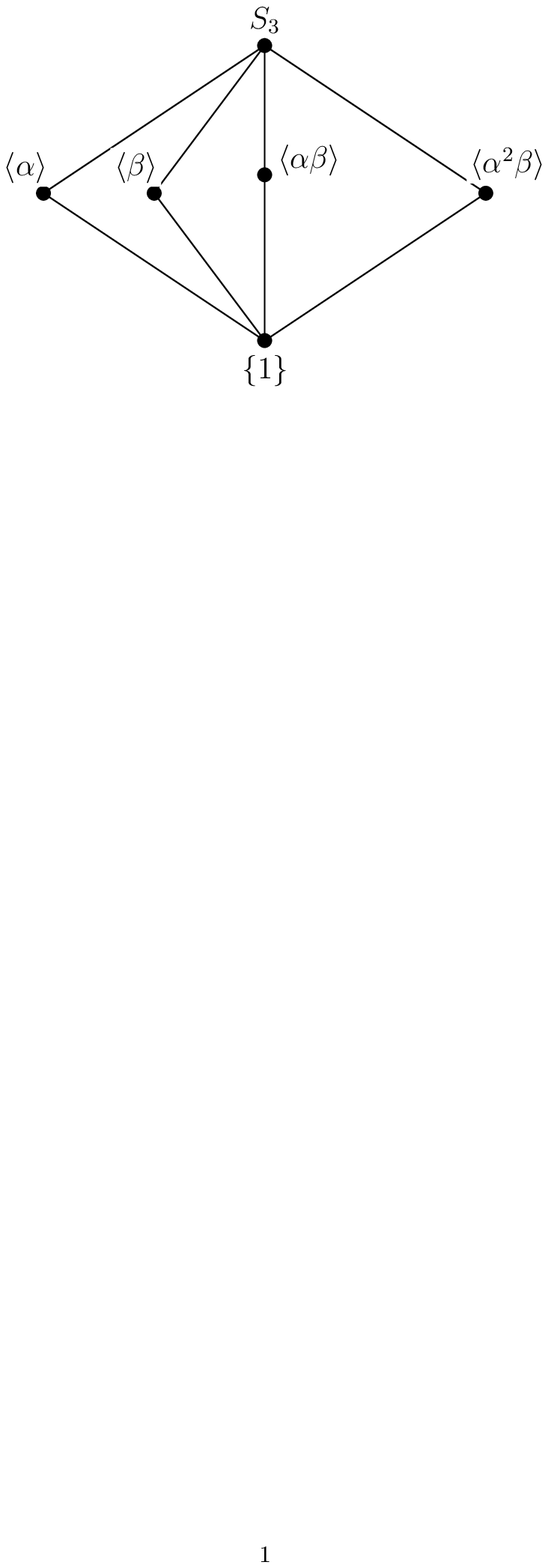} \caption{\small{Hasse diagram of $S_3$}}
\end{center}
\end{figure}

$$\frac{s(S_3)-1}{2}=s^{\star}(\{1\})+s^{\star}(\langle \alpha \rangle)+s^{\star}(\langle \beta \rangle)+s^{\star}(\langle \alpha\beta \rangle)+
s^{\star}(\langle \alpha^2\beta \rangle).$$
Therefore $s(S_3)=19$.
\end{example}

\section{Distinct fuzzy subgroup commutativity for some of dihedral groups }\label{s7}
\begin{remark}\label{r71}
We count distinct fuzzy subgroups of a finite group $G$ on its Hasse diagram for identity cases following:\\
Left to right on subgroups chains increasingly.
\end{remark}
Let $G$ be a finite group. First of all, remark that the distinct fuzzy subgroup commutativity degree $sd(G)$ satisfies
the following relation:
$$
0 < sd(G)\leq 1.
$$
Obviously, the equality $sd(G)=1$ holds if and only if all distinct fuzzy subgroups of $G$ are permutable.\\
Next, for every fuzzy subgroup $\mu$ of $G$, let us denote by $C(H)$ the set consisting of all distinct fuzzy subgroups of $G$
which commute with $\mu$, that is
$$
C(H)=\{\nu\in S(G) | \nu\in P(\mu)\}.
$$
Then
$$
sd(G)=\frac{1}{|S(G)|^2}\sum_{\mu\in S(G)}|C(\mu)|.
$$
It is clear that the fuzzy normal subgroups of $G$ are contained in each set $C(\mu)$(see \cite{MKA}), which implies that
$$
\frac{|N(G)|}{|S(G)|} \leq sd(G)
$$
such that a remarkable modular sublattice of $S(G)$ is the distinct fuzzy normal subgroup lattice $N(G)$, which consists of all distinct fuzzy normal subgroups of $G$.\\
Note that we have $sd(G)=\frac{|N(G)|}{|S(G)|}$ if and only if $N(G)=S(G)$.\\
By \ref{c4}, It is clear that the fuzzy quasinormal subgroups of $G$ are contained in each set $C(\mu)$(see \cite{MKA}), which implies that
$$
\frac{|QN(G)|}{|S(G)|} \leq sd(G)
$$
such that a remarkable modular sublattice of $S(G)$ is the distinct fuzzy quasinormal subgroup lattice $QN(G)$, which consists of all distinct fuzzy normal subgroups of $G$.\\
Note that we have $sd(G)=\frac{|QN(G)|}{|S(G)|}$ if and only if $QN(G)=S(G)$.\\

\begin{example}
Let $S_3=\langle \alpha, \beta | \alpha^3=\beta^2=(\alpha\beta)^2=1 \rangle$.

Let $\mathcal{A}_1$ be the set of all distinct nontrivial fuzzy subgroups of $S_3$ such that are in Hasse subdiagram of $S_3$(chain $\{1\} \subset \langle \alpha \rangle
\subset S_3$, such that its support is exactly $\langle \alpha \rangle$)following:
\begin{figure}[]
\begin{center}
\includegraphics[angle=0, scale=0.5]
{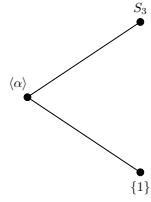} \caption{\small{Hasse subdiagram of $S_3$}}
\end{center}
\end{figure}
Let $\mathcal{A}_2$ be the set of all distinct  nontrivial fuzzy subgroups of $S_3$ on the chain $\{1\} \subset \langle \beta \rangle
\subset S_3$ of $S_3$, such that its support is exactly $\langle \beta \rangle$,\\
$\mathcal{A}_3$ be the set of all distinct nontrivial fuzzy subgroups of $S_3$ on the chain $\{1\} \subset \langle \alpha\beta \rangle
\subset S_3$ of $S_3$, such that its support is exactly $\langle \alpha\beta \rangle$,\\
and $\mathcal{A}_4$ be the set of all distinct nontrivial fuzzy subgroups of $S_3$ on the chain $\{1\} \subset \langle \alpha^2\beta \rangle
\subset S_3$ of $S_3$, such that its support is exactly $\langle \alpha^2\beta \rangle$.\\
Thus $|\mathcal{A}_1|=|\mathcal{A}_2|=|\mathcal{A}_3|=|\mathcal{A}_4|=1$.(For details see \cite{MKA}).
It is clear that for every two subgroups $H$ and $K$ of $S_3$, their product $HK=\{hk | h\in H, k\in K\}$ is a subgroup in $S_3$ except $\langle \beta \rangle\langle \alpha\beta \rangle$,
$\langle \beta \rangle\langle \alpha^2\beta \rangle$, $\langle \alpha\beta \rangle\langle \beta \rangle$, $\langle \alpha\beta \rangle\langle
\alpha^2\beta \rangle$, $\langle \alpha^2\beta \rangle\langle \beta \rangle$ and $\langle \alpha^2\beta \rangle\langle \alpha\beta \rangle$.
If $\mu$ be a fuzzy subgroup of $S_3$ such that its support is exactly equal to $\langle \beta \rangle$
and $\nu$ be a fuzzy subgroup of $S_3$ such that its support is exactly equal to $\langle \alpha\beta \rangle$ or $\langle \alpha^2\beta \rangle$ then $\mu$ and $\nu$ are not permutable. If $\mu$ be a fuzzy subgroup of $S_3$ such that its support is exactly equal to $\langle \alpha\beta \rangle$
and $\nu$ be a fuzzy subgroup of $S_3$ such that its support is exactly equal to $\langle \beta \rangle$ or $\langle \alpha^2\beta \rangle$ then $\mu$ and $\nu$ are not permutable and if $\mu$ be a fuzzy subgroup of $S_3$ such that its support is exactly equal to $\langle \alpha^2\beta \rangle$
and $\nu$ be a fuzzy subgroup of $S_3$ such that its support is exactly equal to $\langle \beta \rangle$ or $\langle \alpha\beta \rangle$ then $\mu$ and $\nu$ are not permutable. Therefore by theorem \ref{t53} and proposition \ref{p51},
$s^{\star}(\langle \beta \rangle)=s^{\star}(\langle \alpha\beta \rangle)=s^{\star}(\langle \alpha^2\beta \rangle)=s^{\star}(\langle \alpha \rangle)=2$. Thus\\
$sd(S_3)=\frac{1}{|S(S_3)|^2}(|\mathcal{A}_1||S(S_3)|+
|\mathcal{A}_2|(|\mathcal{A}_1|+|\mathcal{A}_2|+1)+
|\mathcal{A}_3|(|\mathcal{A}_1|+|\mathcal{A}_3|+1)+
|\mathcal{A}_4|(|\mathcal{A}_1|+|\mathcal{A}_4|+1)+
s^{\star}(\{1\})|S(S_3)|).$
So by example \ref{e65}, $sd(S_3)=\frac{50}{361}$.
\end{example}

\begin{example}
Let $D_8=\langle \alpha,\beta | \alpha^4=\beta^2=1, \beta\alpha\beta=\alpha^{-1} \rangle$, it is clear that for every two subgroups
$H$ and $K$ of $D_8$, their product $HK=\{hk | h\in H, k\in K\}$ is a subgroup in $D_8$ except $\langle \beta \rangle\langle \alpha\beta \rangle$,
$\langle \beta \rangle\langle \alpha^{-1}\beta \rangle$, $\langle \alpha^2\beta \rangle\langle \alpha\beta \rangle$, $\langle \alpha^2\beta \rangle\langle \alpha^{-1}\beta \rangle$, $\langle \alpha\beta \rangle\langle \beta \rangle$, $\langle \alpha\beta \rangle\langle \alpha^2\beta \rangle$, $\langle \alpha^{-1}\beta \rangle\langle \beta \rangle$ and $\langle \alpha^{-1}\beta \rangle\langle \alpha^2\beta \rangle$. If $\mu$ be a fuzzy subgroup of $D_8$ such that its
support is exactly equal to $\langle \beta \rangle$ or $\langle \alpha^2\beta \rangle$ and $\nu$ be a fuzzy subgroup of $D_8$ such that its
support is exactly equal to $\langle \alpha\beta \rangle$ or $\langle \alpha^{-1}\beta \rangle$ then $\mu$ and $\nu$ are not permutable.
Let $\mathcal{A}_1$ be the set of all distinct  nontrivial fuzzy subgroups of $D_8$ on the chain $\{1\} \subset \langle \beta \rangle
\subset \langle \alpha^2, \beta \rangle \subset D_8$ of $D_8$, such that its support is exactly $\langle \alpha^2, \beta \rangle$ or $\langle \beta \rangle$,\\
$\mathcal{A}_2$ be the set of all distinct nontrivial fuzzy subgroups of $D_8$ on the chain $\{1\} \subset \langle \alpha^2\beta \rangle
\subset \langle \alpha^2, \beta \rangle \subset D_8$ of $D_8$, such that its support is exactly $\langle \alpha^2, \beta \rangle$ or $\langle \alpha^2\beta \rangle$,\\
$\mathcal{A}_3$ be the set of all distinct nontrivial fuzzy subgroups of $D_8$ on the chain $\{1\} \subset \langle \alpha^2 \rangle
\subset \langle \alpha^2, \beta \rangle \subset D_8$ of $D_8$, such that its support is exactly $\langle \alpha^2, \beta \rangle$ or $\langle \alpha^2 \rangle$,\\
$\mathcal{A}_4$ be the set of all distinct  nontrivial fuzzy subgroups of $D_8$ on the chain $\{1\} \subset \langle \alpha^2 \rangle
\subset \langle \alpha \rangle \subset D_8$ of $D_8$, such that its support is exactly $\langle \alpha \rangle$ or $\langle \alpha^2 \rangle$,\\
$\mathcal{A}_5$ be the set of all distinct nontrivial fuzzy subgroups of $D_8$ on the chain $\{1\} \subset \langle \alpha^2 \rangle
\subset \langle \alpha^2, \alpha\beta \rangle \subset D_8$ of $D_8$, such that its support is exactly $\langle \alpha^2, \alpha\beta \rangle$ or $\langle \alpha^2 \rangle$,\\
$\mathcal{A}_6$ be the set of all distinct nontrivial fuzzy subgroups of $D_8$ on the chain $\{1\} \subset \langle \alpha\beta \rangle
\subset \langle \alpha^2, \alpha\beta \rangle \subset D_8$ of $D_8$, such that its support is exactly $\langle \alpha^2, \alpha\beta \rangle$ or $\langle \alpha\beta \rangle$,\\
and $\mathcal{A}_7$ be the set of all distinct nontrivial fuzzy subgroups of $D_8$ on the chain $\{1\} \subset \langle \alpha^{-1}\beta \rangle
\subset \langle \alpha^2, \alpha\beta \rangle \subset D_8$ of $D_8$, such that its support is exactly $\langle \alpha^2, \alpha\beta \rangle$ or $\langle \alpha^{-1}\beta \rangle$.\\
Clearly, $|\mathcal{A}_1|=|\mathcal{A}_2|=|\mathcal{A}_3|=|\mathcal{A}_5|=|\mathcal{A}_6|=|\mathcal{A}_7|=2+8-1=9$ and $|\mathcal{A}_4|=2+4-1=5$.
(For details see \cite{MKA}).\\
Thus by example \ref{e62}\\
$sd(D_8)=\frac{1}{63^2}(|\mathcal{A}_1|(s(D_8)-|\mathcal{A}_6|-|\mathcal{A}_7|)+
|\mathcal{A}_2|(s(D_8)-|\mathcal{A}_6|-|\mathcal{A}_7|)+
|\mathcal{A}_3|s(D_8)+|\mathcal{A}_4|s(D_8)+|\mathcal{A}_5|s(D_8)+
|\mathcal{A}_6|(s(D_8)-|\mathcal{A}_1|-|\mathcal{A}_2|)+
|\mathcal{A}_7|(s(D_8)-|\mathcal{A}_1|-|\mathcal{A}_2|)+s^{\star}(\{1\})s(S_8)).
$\\
So $sd(D_8)=\frac{3897}{3969}$.
\end{example}

\begin{proposition}
suppose that $p$ is a prime and $p\geq 3$. If $G$ is the dihedral
group of order $2p$, then $sd(G)=1$.
\end{proposition}

\begin{proof}
It is clear that for every two subgroups
$H$ and $K$ of $D_{2p}$, $HK=\{hk | h\in H, k\in K\}$ is a subgroup in $D_{2p}$. If $\mu$ be a fuzzy subgroup of $D_{2p}$ such that its support is exactly equal to $H$
and $\nu$ be a fuzzy subgroup of $D_{2p}$ such that its support is exactly equal to $K$, then by theorem \ref{t31}, proposition \ref{p32} and theorem \ref{t21}, $\mu$ and $\nu$ are permutable.So that $sd(G)=1$.
\end{proof}

\section*{Acknowledgment}
The authors are grateful to the reviewers for their remarks which improve the previous version of the paper. The authors are deeply grateful to the Islamic Azad University, Ashtian Branch for the financial support for the proposal entitled "FUZZY SUBGROUPS COMMUTATIVITY DEGREE OF DIHEDRAL GROUPS".


\end{document}